\documentclass[a4paper,12pt]{amsart}
\usepackage[utf8]{inputenc}
\usepackage[T1]{fontenc}
\usepackage{mathptmx}
\usepackage{amsfonts}

\usepackage[english]{babel}
\usepackage{amssymb,amsmath}
\usepackage{xcolor}
\usepackage{hyperref}
\usepackage{hyperref}
\newcommand{\cB}{\mathcal B}
\newcommand{\cG}{\mathcal G}

\newcommand{\be}{{\mathbf e}}
\newcommand{\ben}{{\be^*_n}}

\newcommand{\cLL}{{\mathbf {L}}_m^{\rm ch}}
\newcommand{\cL}{{\mathbf {L}}_m^{{\rm ch}, t}}
\newcommand{\ch}{{\mathfrak{CG}}}
\newcommand{\cht}{{\mathfrak{CG}}^t}
\newcommand{\bL}{{\mathbf L}_m}

\newcommand{\bq}{{\mathbf q}}

\newcommand{\tbe}{{\tilde\be}}

\newcommand {\SC} {{\mathbb C}}

\newcommand {\SG} {{\mathbb G}}

\newcommand {\SK} {{\mathbb K}}
\newcommand {\SN} {{\mathbb N}}
\newcommand {\SR} {{\mathbb R}}
\newcommand {\ST} {{\mathbb T}}

\newcommand {\SX} {{\mathbb X}}

\newcommand {\SZ} {{\mathbb Z}}
\newcommand {\tphi} {{\tilde\varphi}}
\newcommand {\tmu} {{\tilde\mu}}
\newcommand {\tg} {{\tilde{g}}}

\newcommand {\al} {{\alpha}}
\newcommand {\dt} {{\delta}}

\newcommand {\e} {{\varepsilon}}
\newcommand {\ga} {{\gamma}}
\newcommand {\Ga} {{\Gamma}}

\newcommand {\la} {{\lambda}}
\newcommand {\La} {{\Lambda}}
\newcommand {\bLa} {{{\bar{\Lambda}}}}

\newcommand{\fN}{\mathfrak{N}}

\def\supp{\mathop{\rm supp}}
\def\dist{\mathop{\rm dist}}
\def\sig{\mathop{\rm sign}}

\numberwithin{equation}{section}
\newtheorem{theorem}{Theorem}[section]
\newtheorem{lemma}[theorem]{Lemma}

\newtheorem{Remark}[theorem]{Remark}
\newtheorem{proposition}[theorem]{Proposition}

\newtheorem{example}[theorem]{Example}

\newcommand {\ProofEnd} {
\begin{flushright} \vskip -0.2in $\Box$ \end{flushright}}

\newcommand{\Ba}[1]{\begin{array}{#1}}
\newcommand{\Ea}{\end{array}}
\newcommand{\Bd}{\begin{description}}
\newcommand{\Ed}{\end{description}}
\newcommand{\Be}{\begin{equation}}
\newcommand{\Ee}{\end{equation}}
\newcommand{\Bea}{\begin{eqnarray}}
\newcommand{\Eea}{\end{eqnarray}}
\newcommand{\Beas}{\begin{eqnarray*}}
\newcommand{\Eeas}{\end{eqnarray*}}
\newcommand{\Benu}{\begin{enumerate}}
\newcommand{\Eenu}{\end{enumerate}}
\newcommand{\Bi}{\begin{itemize}}
\newcommand{\Ei}{\end{itemize}}

\newcommand{\BR}{\begin{Remark} \em}
\newcommand{\ER}{\end{Remark}}
\newcommand{\BE}{\begin{example} \em}
\newcommand{\EE}{\end{example}}

\newcommand {\Ds} {\displaystyle}

\newcommand {\mand} {{\quad\mbox{and}\quad}}
\renewcommand {\mid} {{\,\,\,\colon\,\,\,}}

\newcommand{\bline}{{\bigskip

\noindent}}

\newcommand{\sline}{{\smallskip

\noindent}}

\newcommand {\bone} {{\bf 1}}

\newcommand{\cGtm}{{\SG_m^{{\rm ch},t}}}

\renewcommand {\span} {\mbox{\rm span}\,}
\newcommand{\sgn}{\mathrm{sign \,}}
\newcommand{\cupdot}{\mathbin{\mathaccent\cdot\cup}}

\newcounter{reg}
\begin{document}

\title[Lebesgue inequalities for Chebyshev Greedy Algorithms]{Lebesgue inequalities for  Chebyshev Thresholding Greedy Algorithms}%
\author{P. M. Bern\'a}
\address{Pablo M. Bern\'a
\\
Departamento de Matem\'aticas
\\
Universidad Aut\'onoma de Madrid
\\
28049 Madrid, Spain} \email{pablo.berna@uam.es}

\author{\'O. Blasco}

\address{\'Oscar Blasco
\\
Department of Analysis Mathematics \\ Universidad de Valencia, Campus de
Burjassot\\ Valencia, 46100, Spain}
\email{oscar.blasco@uv.es}

\author{G. Garrig\'os}
\address{Gustavo Garrig\'os
\\
Departamento de Matem\'aticas
\\
Universidad de Murcia
\\
30100 Murcia, Spain} \email{gustavo.garrigos@um.es}

\author{E. Hern\'andez}
\address{Eugenio Hern\'andez
\\
Departamento de Matem\'aticas
\\
Universidad Aut\'onoma de Madrid
\\
28049 Madrid, Spain} \email{eugenio.hernandez@uam.es}

\author{T. Oikhberg}
\address{Timur Oikhberg
	\\
	Department of Mathematics
	\\
University of Illinois Urbana-Champaign
	\\
	Urbana, IL 61807, USA} \email{oikhberg@illinois.edu}

\thanks{The research of the first, third and fourth authors are partially supported by the grants MTM-2016-76566-P (MINECO, Spain) and 19368/PI/14 (\emph{Fundaci\'on S\'eneca}, Regi\'on de Murcia, Spain). Also, the first author is supported by a PhD Fellowship from the program "Ayudas para contratos predoctorales para la formaci\'on de doctores 2017" (MINECO, Spain). The second author is supported by Grant MTM-2014-53009-P (MINECO, Spain). Third author partially supported by grant MTM2017-83262-C2-2-P (Spain). The fourth author has received funding from the European Union's Horizon 2020 research and innovation programme under the Marie Sk\l odowska-Curie grant agreement No 777822}



\begin{abstract}
We establish estimates for the Lebesgue parameters of the Chebyshev Weak Thresholding  Greedy Algorithm in the case of general bases in Banach spaces. These generalize and slightly improve earlier results in \cite{DKO}, and are complemented with examples showing the optimality of the bounds. Our results also correct certain bounds recently announced in \cite{SY}, and answer some questions left open in that paper.
\end{abstract}

\maketitle

\section{Introduction }
\setcounter{equation}{0}\setcounter{footnote}{0}
\setcounter{figure}{0}

Let $\mathbb X$ be a Banach space over $\SK = \SR$ or $\SC$, let $\SX^*$ be its dual space, and consider a system $\lbrace \be_n, \ben\rbrace_{n=1}^\infty\subset \SX\times\SX^*$ with the following properties:
\begin{itemize}
\item[a)] $0<\inf_n \lbrace \Vert \be_n\Vert, \Vert \ben\Vert\rbrace\leq \sup_n\lbrace \Vert \be_n\Vert, \Vert \ben\Vert\rbrace<\infty$
\item[b)] $\ben(\be_m) = \delta_{n,m}$, for all $n,m\geq 1$
\item[c)] $\SX= \overline{\span\lbrace \be_n : n\in\mathbb N\rbrace}$
\item [d)] $\SX^* = \overline{\span\lbrace \ben : n\in\mathbb N\rbrace}^{w^*}$.
\end{itemize}

Under these conditions $\mathcal B=\lbrace \be_n\rbrace_{n=1}^\infty$ is called a \textit{seminormalized Markushevich basis for $\SX$} (or M-basis for short), with \textit{dual system} $\lbrace \ben\rbrace_{n=1}^\infty$.
Sometimes we shall consider the following special cases

\begin{itemize}
\item[e)] $\mathcal B$ is a \textit{Schauder basis} if $K_b:=\sup_N \Vert S_N\Vert<\infty$, where $S_Nx:=\sum_{n=1}^N \be^*_n(x)\be_n$ is the $N$-th partial sum operator
\end{itemize}
\begin{itemize}
\item[f)] $\mathcal B$ is a \textit{Ces\`aro basis} if $\sup_N \Vert F_N\Vert<\infty$, where
$F_N:= \frac{1}{N}\sum_{n=1}^N S_n$ is the $N$-th Ces\`aro operator. In this case we use the constant
\Be
\beta=\;\max\big\{\sup_N\|F_N\|,\,\sup_N\|I-F_N\|\big\}.
\label{beta}\Ee
\end{itemize}
With every $x\in\mathbb X$, we shall associate the formal series $x\sim \sum_{n=1}^{\infty} \ben(x)\be_n$, where a)-c) imply that $\lim_{n}\ben(x)=0$. As usual, we denote $\supp x=\lbrace n\in\mathbb N : \ben(x)\neq 0\rbrace$.

\

We recall standard notions about (weak) greedy algorithms; see e.g. the texts \cite{Tem1, Tem15} for details and  historical background.
Fix $t\in(0,1]$. We say that $A$ is a \textit{$t$-greedy set for $x$ of order $m$}, denoted $A\in G(x, m, t)$,  if
$\vert A\vert=m$ and
\Be\min_{n\in A}\vert \ben(x)\vert \geq t\cdot\max_{n\not\in A}\vert\ben(x)\vert.\label{tineq}\Ee
A \textit{$t$-greedy operator of order $m$} is any mapping $\mathcal G_m^t: \SX\to \SX$ which at each $x\in\SX$ takes the form
$$\mathcal G_m^t(x)=\sum_{n\in A}\ben(x)\be_n,\quad \text{for some set}\quad A=A(x,\cG^t_m)\in G(x, m, t).$$
We write $\SG_m^t$ for the set of all $t$-greedy operators of order $m$.
The approximation scheme which assigns a sequence $\{\mathcal G_m^t(x)\}_{m=1}^\infty$ to each vector $x\in\SX$ is called a \textit{Weak Thresholding Greedy Algorithm} (WTGA), see \cite{KT2,TemW}. When $t=1$ one just says Thresholding Greedy Algorithm (TGA), and drops the super-index $t$, that is $\mathcal G_m^1 = \mathcal G_m$, etc.

\

It is standard to quantify the efficiency of these algorithms, among all possible $m$-term approximations, in terms of \emph{Lebesgue-type inequalities}.
That is, for each $m=1, 2,...$, we look for the smallest constant $\bL^t$ such that
\begin{eqnarray}\label{weakin}
\Vert x-\mathcal G_m^t(x)\Vert \leq \bL^t\sigma_m(x),\quad \forall\; x\in\SX, \quad\forall\; \mathcal G_m^t\in \SG_m^t,
\end{eqnarray}
where
$$\sigma_m(x):=\inf\Big\lbrace \Big\Vert x-\sum_{n\in B}b_n \be_n\Big\Vert\mid b_n\in\SK,\quad \vert B\vert \leq m\Big\rbrace.$$
We call the number $\bL^t$ the \emph{Lebesgue parameter} associated with the WTGA, and we just write $\bL$ when $t=1$. We refer to \cite[Chapter 3]{Tem15} for a survey on such inequalities, and to \cite{GHO,DKO,AA1,BBG,BBGHO}  for recent results. It is known that $\bL^t=O(1)$ holds for a fixed $t$ if and only if it holds for all $t\in(0,1]$, and if and only if $\cB$ is unconditional and democratic; see \cite{KT} and \cite[Thm 1.39]{Tem1}. In this special case $\cB$ is called a \emph{greedy basis}.



\

In this paper we shall be interested in \textit{Chebyshev thresholding  greedy algorithms}. These were introduced by Dilworth, Kalton and Kutzarova, see \cite[\S3]{DKK}, as an enhancement of the TGA. Here, we use the weak version considered in \cite{DKO}. Namely, for fixed $t\in(0,1]$ we say that $\ch_m^t:\SX\to\SX$ is a  \textit{Chebyshev $t$-greedy operator} of order $m$ if  for every $x\in\SX$ the set $A=\supp\ch_m^t(x)\in G(x, m, t)$ and moreover
$$\Vert x-\ch_m^t(x)\Vert = \min\Big\lbrace \big\Vert x-\sum_{n\in A}a_n\be_n\big\Vert\mid a_n\in\SK\Big\rbrace.$$
Finally, we define the \textit{weak Chebyshevian Lebesgue parameter} $\cL$ as the smallest constant such that
$$\Vert x-\ch_m^t(x)\Vert \leq \cL\sigma_m(x),\quad \forall\; x\in\SX,\quad \forall \;\ch_m\in \cGtm\,,$$
where $\cGtm$ is the collection of all Chebyshev $t$-greedy operators of order $m$. As before, when $t=1$ we shall omit the index $t$, that is $\mathbf{L}_m^{{\rm ch}}:=\mathbf{L}_m^{{\rm ch},1}$.

When  $\mathbf{L}_m^{{\rm ch}}=O(1)$ the system $\mathcal B$ is called semi-greedy; see \cite{DKK}.
We remark that the first author recently established that a Schauder basis $\cB$ is semi-greedy if and only if is quasi-greedy and democratic; see \cite{B}.

\

In this paper we shall be interested in quantitative bounds of $\cL$ in terms of the quasi-greedy and democracy parameters of a general M-basis $\cB$.
Earlier bounds were obtained by Dilworth, Kutzarova and Oikhberg in \cite{DKO} when $\cB$ is a quasi-greedy basis, and very recently, some improvements were also announced by C. Shao and P. Ye in \cite[Theorem 3.5]{SY}. Unfortunately, various arguments in the last paper seem not to be correct, so one of our goals here is to
give precise statements and proofs for the results in \cite{SY}, and also settle some of the questions which are left open there.

\

To state our results, we recall the definitions of the involved parameters.
Given a finite set $A\subset\SN$, we shall use the following standard notation for the indicator sums:
$$\bone_A = \sum_{n\in A}\be_n\mand \bone_{\varepsilon A}=\sum_{n\in A}\varepsilon_n \be_n,\quad\e\in\Upsilon$$
where $\Upsilon$ is the set of all $\varepsilon = \lbrace\varepsilon_n\rbrace_n\subset\SK$ with $\vert\varepsilon_n\vert=1$. Similarly, we write
\[
P_A(x)=\sum_{n\in A}\ben(x)\be_n.
\]
The relevant parameters for this paper are the following:
\begin{itemize}
\item Conditionality parameters:
$$k_m := \sup_{\vert A\vert \leq m}\Vert P_A\Vert\mand k_m^c=\sup_{\vert A\vert \leq m}\Vert I-P_A\Vert.$$
\item Quasi-greedy parameters: 
$$g_m := \sup_{\mathcal G_k\in \SG_k,\,k\leq m}\,\|\cG_k\|
\mand g_m^c := \sup_{\mathcal G_k\in \SG_k,\,k\leq m}\,\|I-\cG_k\|.
$$
Below we shall also use the variant
\[
\tg_m:=\sup_{{\cG'<\cG}\atop{\cG\in\SG_k,\;k\leq m}}\|\cG-\cG'\|,
\]
where $\cG'<\cG$ means that $A(x,\cG')\subset A(x,\cG)$ for all $x$; see \cite{BBG}.

\item Super-democracy parameters:
$$\tilde{\mu}_m = \sup_{\underset{\vert\varepsilon\vert=\vert\eta\vert=1}{\vert A\vert=\vert B\vert\leq m}}\dfrac{\Vert \bone_{\varepsilon A}\Vert}{\Vert \bone_{\eta B}\Vert}\mand \tilde{\mu}_m^d=\sup_{\underset{\vert\varepsilon\vert=\vert\eta\vert=1}{\vert A\vert=\vert B\vert\leq m, \; A\cap B=\emptyset}}\dfrac{\Vert \bone_{\varepsilon A}\Vert}{\Vert \bone_{\eta B}\Vert}.$$

\item[$\bullet$] Quasi-greedy parameters for constant coefficients (see \cite[(3.11)]{BBG})\[
\ga_m=\sup_{\underset{B\subset A,\;\vert A\vert\leq m}{\vert\varepsilon\vert=1}}\dfrac{\Vert \bone_{\varepsilon B}\Vert}{\Vert \bone_{\e A}\Vert}.
\]
\end{itemize}
Note that $\ga_m\leq g_m\leq \tg_m\leq 2g_m$, but in general $\ga_m$ may be much smaller than $g_m$; see e.g. \cite[\S5.5]{BBG}. Likewise, in \S5 below we show that $\tmu^d_m$ may be much smaller than $\tmu_m$, except for Schauder bases in which both quantities turn out to be equivalent; see Theorem \ref{th_mudN}.

\

Our first result is a general upper bound, which improves and extends \cite[Theorem 2.4]{SY}.

\begin{theorem}\label{main2}
Let $\mathcal B$ be an M-basis in $\mathbb X$, and let $\mathfrak{K}=\sup_{n,j}\Vert \ben\Vert\Vert \be_j\Vert$. Then,
\Be
\cL \leq\, 1\,+\,(1+\tfrac1t)\,\mathfrak{K}\,m\,, \quad \forall\;m\in\SN,\;\; t\in(0,1].\label{ineq_th1}\Ee
Moreover, there exists a pair $(\SX,\cB)$ where the equality is attained for all $m$ and $t$.
\end{theorem}

The second result is a slight generalization of \cite[Theorem 4.1]{DKO},
and gives a correct version of  \cite[Theorem 3.5]{SY}.

\begin{theorem}\label{main1}
Let $\mathcal B$ be an M-basis in $\mathbb X$. Then, for all $m\geq 1$ and $t\in (0,1]$,
\Be\cL \;\leq \;g_{2m}^c\,+\,\frac2t\,\min\big\{\,\tg_m\tmu_m\,,\;\ga_{2m}\tg_{2m}\tmu^d_m\,
\big\}.\label{Lmain}\Ee
\end{theorem}

Our next result concerns lower bounds for $\cL$, for which we need to introduce weaker versions of the democracy parameters with an additional  separation condition. For two finite sets $A,B\subset\SN$ and $c\geq1$, the notation $A>cB$ will stand for $\min A >c\max B$.

\Bi\item Given an integer $c\geq2$, we define
\begin{equation} \label{vmc}
\vartheta_{m,c} := \sup\left\lbrace \dfrac{\Vert \bone_{\varepsilon A}\Vert}{\Vert \bone_{\eta B}\Vert}\mid 
\vert \varepsilon\vert=\vert\eta\vert=1,\;|A|=|B|\leq m\;\mbox{ with $A>cB$ or $B>cA$} \right\rbrace.
\end{equation}
\Ei

\begin{theorem}\label{th3}
If $\cB$ is a Ces\`aro basis in $\SX$ with constant $\beta$, then for every $c\geq 2$
$$\cL\geq \frac{1}{t\beta^2}\frac{c-1}{c+1}\vartheta_{m,c},\quad \forall\;m\in\SN,\;t\in(0,1].$$
\end{theorem}

We shall also establish, in Theorem
\ref{Thm_thetam} below, a similar lower bound valid for more general M-bases (not necessarily of Ces\`aro type),
in terms of a new parameter $\theta_m$ which is invariant under rearrangements of $\cB$.  

\begin{Remark}
{\rm One may compare the bounds for $\cLL$ above with those  for $\bL$ given in \cite{BBG}
\[
(1)\;\bL\leq 1+3\frak{K}m,\qquad (2)\;
\bL\leq k^c_{2m}+\tg_m\tmu_m,\quad\mand\;(3)\;
\bL\geq \tmu_m^d,
\]
which illustrate a slightly better behavior of the Chebishev TGA.
Observe that one also has the trivial inequalities
\[
\cL\leq \bL^t\leq k^c_m\,\cL.
\]
Indeed, $\cL \leq \bL^t$ is direct by definition, while $\bL^t \leq k_m^c \cL$ can be proved as follows: take $x\in\mathbb X$ and $A=\supp \cG^t_m(x)$. Pick a Chebyshev greedy operator $\ch^t_m$ such that $\supp \ch^t_m (x)=A$. Then
\[
\Vert x-\mathcal{G}_m^t(x)\Vert = \|(I-P_A)x\|=\Vert (I-P_A)(x-\mathfrak{CG}_m^t(x))\Vert \leq k_m^c\Vert x-\mathfrak{CG}_m^t(x)\Vert,
\]
so $\bL^t\leq k_m^c \cL$. Hence, when $\mathcal B$ is unconditional
then $\bL^t \approx \cL$. However for all conditional quasi-greedy and democratic bases we have $\cLL=O(1)$, but $\bL\to\infty$.}
\end{Remark}

The paper is organized as follows. Section \ref{previous} is devoted to preliminary lemmas.  In Section \ref{proof} we prove Theorems \ref{main2},   \ref{main1} and \ref{th3}, and also establish the more general lower bound in Theorem \ref{Thm_thetam}, giving various situations in which it applies.
Section \ref{examples} is devoted to examples illustrating  the optimality of the results; in particular, an optimal bound of $\cLL$ for the trigonometric system in $L^1(\ST)$, settling a question left open in \cite{SY}.
In Section \ref{s:comparison_mu_mu_d} we investigate the equivalence between $\tmu^d_m$ and $\tmu_m$
and show Theorem \ref{th_mudN}. Finally, in Section 6 we study the convergence 
of $\ch_m (x)$ and $\cG_m(x)$ to $x$ under the \emph{strong} M-basis assumption, settling a gap in \cite{SY,Wo}.

\section{Preliminary results}\label{previous}

We recall some basic concepts and results that will be used later in the paper; see \cite{DKK,BBG}.
For each $\al>0$  we define the $\al$-truncation of a scalar $y\in\mathbb K$  as

$$T_\al(y)=\al\, \sgn y\; \mbox{ if }\, \vert y\vert \geq \al,\mand T_\al(y)=y\; \mbox{ if }\, \vert y\vert\leq\al.$$
We extend $T_\al$ to an operator in $\mathbb X$ by formally assigning $T_\al(x)\sim\sum_{n=1}^\infty T_\al(\ben(x))\be_n$, that is
$$T_\al(x):=
 \al \bone_{\varepsilon \Lambda_\al(x)}+(I-P_{\Lambda_\al(x)})(x),$$
where $\Lambda_\al(x)=\lbrace n : \vert\ben(x)\vert>\al\rbrace$ and $\varepsilon=\lbrace\sgn(\ben(x))\rbrace$. Of course, this operator is well defined since $\Lambda_\al(x)$ is a finite set.
In \cite{BBG} we can find the following result:
\begin{lemma}{\cite[Lemma 2.5]{BBG}}\label{trun1}
For all $\al>0$ and $x\in\mathbb X$, we have
$$\Vert T_\al(x)\Vert \leq g_{\vert\Lambda_\al(x)\vert}^c\Vert x\Vert.$$
\end{lemma}

We also need a well known property from \cite{DKK,DKKT}, formulated as follows.
\begin{lemma}{\cite[Lemma 2.3]{BBG}}\label{propc}
If $x\in\mathbb X$ and $\varepsilon=\lbrace\sgn(\ben(x))\rbrace$, then
\Be\min_{n\in G}\vert\ben(x)\vert \Vert\bone_{\varepsilon G}\Vert \leq \tg_{\vert G\vert} \Vert x\Vert,\quad \forall G\in G(x,m,1).\label{propG}\Ee
\end{lemma}
The following version of \eqref{propG}, valid even if $G$ is not greedy, improves \cite[Lemma 2.2]{DKO}.

\begin{lemma}\label{propc1}
Let $x\in\mathbb X$ and $\varepsilon=\lbrace\sgn(\ben(x))\rbrace$. For every set finite $A\subset\SN$, if $\al=\min_{n\in A}|\be^*_n(x)|$, then
\Be\al \Vert\bone_{\varepsilon A}\Vert \,\leq \,\ga_{|A\cup\La_\al(x)|}\,\tg_{|A\cup\La_\al(x)|} \Vert x\Vert,\label{propA}\Ee
where $\Lambda_\al(x)=\lbrace n : \vert\ben(x)\vert>\al\rbrace$.
\end{lemma}
\begin{proof}
Call $G=A\cup\La_\al(x)$, and notice that it is a greedy set for $x$. Then,
\[
\al \Vert\bone_{\varepsilon A}\Vert \,\leq\,\al\,\ga_{|G|}\|\bone_{\e G}\|\,\leq\,\ga_{|G|}\,\tg_{|G|}\,
\|x\|,\]
using \eqref{propG} in the last step.
\end{proof}

\begin{Remark} \label{propR1}{\rm The following is a variant of \eqref{propA} with a different constant \Be
\min_{n\in A}|\be^*_n(x)|\; \Vert\bone_{\varepsilon A}\Vert \,\leq \,k_{|A|}\, \Vert x\Vert.
\label{propR}\Ee
A similar proof as the one in Lemma \ref{propc1} can be seen in \cite[Proposition 2.5]{BB1}.}
\end{Remark}

Finally, using convexity as in \cite[Lemma 2.7]{BBG}, one has the elementary lemma.
\begin{lemma}\label{conv}
For all finite sets $A\subset\SN$ and scalars $a_n\in\SK$ it holds
\[
\Big\|\sum_{n\in A} a_n \be_n\Big\|\leq \,\max_{n\in A}|a_n|\,\sup_{|\e|=1}\big\|\bone_{\e A}\big\|.
\]
\end{lemma}

\section{Proof of the main results}\label{proof}

\subsection{Proof of Theorem \ref{main2}}  Let $x\in\mathbb X$ and $\cht_m\in \cGtm$ be a fixed Chebyshev $t$-greedy operator, and denote by $A=\supp\cht_mx\in G(x,m,t)$. Pick any $z=\sum_{n\in B}b_n\be_n$ such that $\vert B\vert = m$. By definition of the Chebyshev operators,
\begin{eqnarray*}
\Vert x-\mathfrak{CG}_m^t(x)\Vert \leq \Vert x-P_{A\cap B}(x)\Vert\leq \Vert P_{B\setminus A}(x)\Vert + \Vert x-P_B(x)\Vert.
\end{eqnarray*}
On the one hand, using \eqref{tineq},
$$\Vert P_{B\setminus A}(x)\Vert \leq \sup_n\Vert \be_n\Vert\sum_{j\in B\setminus A}\vert \mathbf{e}_j^*(x)\vert\leq \frac{1}{t}\sup_n\Vert\be_n\Vert \sum_{j\in A\setminus B}\vert \mathbf{e}_j^*(x-z)\vert \leq \frac{1}{t}\mathfrak{K}m\Vert x-z\Vert.$$
On the other hand, using the inequality (3.9) of \cite{BBG},
$$\Vert x-P_B(x)\Vert=\Vert (I-P_B)(x-z)\Vert \leq k_m^c\Vert x-z\Vert \leq (1+\mathfrak{K}m)\Vert x-z\Vert.$$
Hence, $\cL\leq 1+\left(1+\frac{1}{t}\right)\mathfrak{K}m$.
Finally, the fact that the equality in \eqref{ineq_th1} can be attained is witnessed by Examples \ref{summing} and \ref{diff} below.

\subsection{Proof of Theorem \ref{main1}}

The scheme of the proof follows the lines in \cite[Theorem 3.2]{DKK} and \cite[Theorem 4.1]{DKO}, with some additional simplifications introduced in \cite{BBG}.

Given $x\in\mathbb X$ and $\cht_m\in \cGtm$, we denote by $A=\supp\cht_mx\in G(x,m,t)$. Pick any $z=\sum_{n\in B}b_n\be_n$ such that $\vert B\vert = m$. By definition of the Chebyshev operators,
\Be\|x-\ch^t_mx\|\leq \|x-p\|,\quad \mbox{for any $p=\sum_{n\in A}a_n\be_n$.}\label{chtp}\Ee
We make the selection of $p$ suggested in \cite{DKK}. Namely, if
$\al=\max_{n\notin A}|\be^*_n(x)|$, we let
\[
p=P_A(x)- P_A\big(T_\al(x-z)\big).
\]
It is easily verified that \Bea
x-p & = & (I-P_A)\big(x-T_\al(x-z)\big)+T_\al(x-z) \nonumber\\ & = & P_{B\setminus A}\big(x-T_\al(x-z)\big)+T_\al(x-z).
\label{p12}\Eea
Since $\La_\al(x-z)=\{n\mid |\be^*_n(x-z)|>\al\}\subset A\cup B$, then Lemma \ref{trun1} gives
\Be
\big\|T_\al(x-z)\big\|\leq g^c_{2m}\|x-z\|.
\label{Taz}\Ee
Next we treat the first term in \eqref{p12}.
Observe that $\max_{n\in B\setminus A}\vert \ben(x-T_\alpha(x-z))\vert\leq 2\alpha$, so Lemma \ref{conv} gives
\Bea
\big\|P_{B\setminus A}\big(x-T_\al(x-z)\big)\big\| & \leq &
2\al\,\sup_{|\e|=1}\big\|\bone_{\e(B\setminus A)}\big\|\nonumber\\
& \leq &
\frac2t\,\min_{n\in A\setminus B}|\be^*_n(x-z)|\,\sup_{|\e|=1}\big\|\bone_{\e(B\setminus A)}\big\|=(*).\label{PBA}
\Eea
At this point we have two possible approaches. Let $\eta_n=\sgn [e^*_n(x-z)]$. In the first approach we pick a greedy set
$\Ga\in G(x-z,|A\setminus B|,1)$, and control  \eqref{PBA} by
\Be
(*)
 \,\leq \,\frac2t\,\min_{n\in \Ga}|\be^*_n(x-z)|\,\,\tmu_m\,\big\|\bone_{\eta\Ga}\big\|\, \leq \, \frac2t\,\tmu_m\,\tg_{m}\|x-z\|,\label{PBA1}\Ee
 using Lemma \ref{propc} in the last step.
In the second approach, we argue as follows
\Be
(*)
 \,\leq \,\frac2t\,\min_{n\in A\setminus B}|\be^*_n(x-z)|\,\,\tmu^d_m\,\big\|\bone_{\eta(A\setminus B)}\big\|\, \leq \, \frac2t\,\ga_{2m}\,\tg_{2m}\,\tmu_m^d\,\|x-z\|,\label{PBA2}\Ee
using in the last step Lemma \ref{propc1}
and the fact that, if $\dt=\min_{A\setminus B}|\be^*_n(x-z)|$, then the set $(A\setminus B)\cup\{n\mid|\be^*_n(x-z)|>\dt\}\subset A\cup B$ and hence has cardinality $\leq 2m$.

\
We can now combine the estimates displayed in \eqref{chtp}-\eqref{PBA2} and obtain
\[
\|x-\ch^t_mx\|\leq \,\big[g^c_{2m}+\frac2t\,\min\big\{\,\tg_m\tmu_m\,,\;\ga_{2m}\tg_{2m}\tmu^d_m\,\big\}\big]\,\|x-z\|,
\]
which after taking the infimum over all $z$ establishes Theorem \ref{main1}.
\ProofEnd

\begin{Remark}
{\rm
In \cite[Theorem 3.5]{SY} a stronger inequality is stated (for $t=1$), namely \Be
\cLL\leq g^c_{2m}+2\tg_m\tmu^d_m.\label{Ye}\Ee
 The proof, however, seems to contain a gap, and a missing factor $k_m^c$ should also appear in the last summand.
 Nevertheless, it is still fair to ask whether the inequality \eqref{Ye} asserted in \cite{SY} may be true  with a different proof.}

\end{Remark}

\begin{Remark}
{\rm Using Remark \ref{propR1} in place of Lemma \ref{propc1} in  \eqref{PBA2}  above leads to an alternative and slightly simpler estimate
\Be
\cL \;\leq \;g_{2m}^c\,+\,\frac2t\,k_m\tmu_m^d\,.\label{cLk}\Ee
However, this would not be as efficient as \eqref{Lmain} when $\cB$ is quasi-greedy and conditional.}
\end{Remark}


\begin{Remark}
{\rm When $\mathcal B$ is quasi-greedy with constant $\bq=\sup_m g_m<\infty$, then Theorem \ref{main1} implies the following
\[
\cL \leq \bq+ 4t^{-1}\, \bq^2 \,\tilde{\mu}^d_m.\]
This is a slight improvement with respect to \cite[Theorem 4.1]{DKO}.
}\end{Remark}

\subsection{Proof of Theorem \ref{th3}}

Recall that $S_N=\sum_{n=1}^N \be^*_n(\cdot)\be_n$ and
\[
F_N(x)=\frac1N\sum_{n=1}^NS_n(x)=\sum_{n=1}^N\big(1-\frac{n-1}N\big)\be^*_n(x)\be_n.
\]
For $M>N$ we define the operators (of de la Vall\'ee-Poussin type) \Bea
V_{N,M}(x) & = & \frac M{M-N}\,F_{M}(x)-\frac N{M-N}F_{N}(x)\nonumber\\
& = &
\sum_{n=1}^N\be^*_n(x)\be_n\,+\,\sum_{n=N+1}^M\big(1-\frac{n-N-1}{M-N}\big)\,\be^*_n(x)\be_n.\label{VNM}\Eea
In particular, observe that, for $\beta$ as in \eqref{beta} we have
\Be
\max\big\{\|V_{N,M}\|, \|I-V_{N,M}\|\big\}\,\leq\, \frac{M+N}{M-N}\,\beta.
\label{VNMb}\Ee
We next prove that, if $c\geq2$, then for all $A,B\subset\SN$ such that $B>cA$ with $|A|=|B|\leq m$ it holds
\Be
\cL\geq \frac1{t\beta}\,\frac{c-1}{c+1}\,\frac{\|\bone_{\e A}\|}{\|\bone_{\eta B}\|},\quad\forall\;|\e|=|\eta|=1.
\label{beta1}\Ee
Pick any set $C>B$ such that $|B\cup C|=m$, and let
\[
x=\bone_{\e A}+t\bone_{\eta B}+t\bone_C.
\]
Then $B\cup C\in G(x,m,t)$, and hence there is a Chebyshev $t$-greedy operator so that
\[
x-\ch^t_m(x)=\bone_{\e A}+\sum_{n\in B\cup C} a_n\be_n,
\]
for some scalars $a_n\in\SK$. Clearly,
\[
\|x-\ch^t_m(x)\|\leq \cL\sigma_m(x)\,\leq\,\cL\,\|t\bone_{\eta B}\|,
\]
using $z=\bone_{\e A}+t\bone_C$ an $m$-term approximant.
On the other hand, let $N=\max A$. Since $\min B\cup C> cN$, then
\eqref{VNM} yields
\[
V_{N,cN}(x-\ch^t_mx)=\bone_{\e A}.
\]
Therefore, \eqref{VNMb} implies that
\[
\|x-\ch^t_m(x)\|\geq \frac{\|V_{N,cN}(x-\ch^t_mx)\|}{\|V_{N,cN}\|}\geq
\frac{c-1}{(c+1)\beta}\,\|\bone_{\e A}\|.
\]
We have therefore proved \eqref{beta1}.

We next show that when $|A|=|B|\leq m$ satisfy $A>cB$ then
\Be
\cL\geq \frac1{t\beta^2}\,\frac{c-1}{c+1}\,\frac{\|\bone_{\e A}\|}{\|\bone_{\eta B}\|},\quad\forall\;|\e|=|\eta|=1.
\label{beta2}\Ee
This together with \eqref{beta1} is enough to establish Theorem \ref{th3}. We shall actually show a slightly
stronger result:

\begin{lemma}
\label{lem_x}
Let $|A|=|B|\leq m$ and let $y\in\SX$ be such that $|y|_\infty:=\sup_n|\be^*_n(y)|\leq1$ and $A>c(B\cupdot \supp y)$. Then
\Be
\cL\geq \frac1{t\beta^2}\,\frac{c-1}{c+1}\,\frac{\|\bone_{\e A}\|}{\|\bone_{\eta B}+y\|},\quad\forall\;|\e|=|\eta|=1.
\label{beta3}\Ee
\end{lemma}

Observe that the case $y=0$ in \eqref{beta3} yields \eqref{beta2}. We now show \eqref{beta3}.
Pick a large integer $\la>1$ and a set $C>\la A$ such that $|B\cup C|=m$. Let
\[
x=\bone_{\e A}+ty+t\bone_{\eta B}+t\bone_C.
\]
As before, $B\cup C\in G(x,m,t)$, and hence for some Chebyshev $t$-greedy operator we have
\[
x-\ch^t_m(x)=\bone_{\e A}+ty+\sum_{n\in B\cup C} a_n\be_n,
\]
for suitable scalars $a_n\in\SK$. Choosing $\bone_{\e A}+t\bone_C$ as $m$-term approximant of $x$ we see that
\[
\|x-\ch^t_m(x)\|\leq \cL\sigma_m(x)\,\leq\,\cL\,t\,\|\bone_{\eta B}+y\|.
\]
On the other hand, calling $N=\max(B\cupdot\supp y)$ and $L=\max A$ we have
\[
(I-V_{N,cN})\circ V_{L,\la L}\big(x-\ch^t_mx\big)=\bone_{\e A}
\]
Thus,
\[
\|x-\ch^t_m(x)\|\geq \frac{\|\bone_{\e A}\|}{\|I-V_{N,cN}\|\|V_{L,\la L}\|}\geq
\frac{c-1}{(c+1)\beta}\,\frac{\la-1}{(\la+1)\beta}\,\|\bone_{\e A}\|.
\]
Therefore we obtain
\[
\cL\geq \frac1{t\beta^2}\,\frac{c-1}{c+1}\,\frac{\la-1}{\la+1}\,\frac{\|\bone_{\e A}\|}{\|\bone_{\eta B}+y\|}
\]
which letting $\la\to\infty$ yields \eqref{beta3}. This completes the proof of Lemma \ref{lem_x}, and hence of Theorem \ref{th3}.

\begin{Remark}
{\rm
When $\cB$ is a Schauder basis, a similar proof gives the following lower bound, which is also obtained in \cite[Theorem 2.2]{SY}
\[
\cL\geq\,\frac1{(K_b+1)t}\,\sup\Big\{\frac{\|\bone_{\e A}\|}{\|\bone_{\eta B}\|}\mid |A|=|B|= m, \;\mbox{$A>B$ or $B>A$ },\;|\e|=|\eta|=1\Big\}.
\]
The statement for Ces\`aro bases, however, will be needed for the applications in \S\ref{ex_trig}.
}\end{Remark}

\subsection{Lower bounds for general M-bases}
Observe that
\[
\vartheta_{m,c}=\sup_{|A|\leq m}\vartheta_c(A),\quad\mbox{where}\quad
\vartheta_c(A)=\sup_{{B\mid |B|=|A|}\atop{{B>cA}\atop{\e,\eta\in\Upsilon}}}
\max\Big\{\frac{\|\bone_{\e A}\|}{\|\bone_{\eta B}\|},\frac{\|\bone_{\eta B}\|}{\|\bone_{\e A}\|}\Big\}.
\]
We consider a new parameter\Be
\vartheta_m=\sup_{|A|\leq m}\;\inf_{c\geq1}\;\vartheta_c(A).
\label{theta2}\Ee
We remark that, unlike $\vartheta_{m,c}$, the  parameter $\vartheta_m$ depends on $\{\be_n\}_{n=1}^\infty$ but not on the reorderings of the system.
We shall give a lower bound for $\cL$ in terms of $\vartheta_m$ in a less restrictive situation than the Ces\`aro basis assumption on $\{\be_n\}_{n=1}^\infty$. 

\

Given $\rho\geq1$, we say that $\{\be_n\}_{n=1}^\infty$ is $\rho$-admissible if the following holds: \emph{for each finite set $A\subset\SN$, there exists $n_0=n_0(A)$ such that, for all sets $B$ with $\min B\geq n_0$ and $|B|\leq|A|$,}
\Be
\big\|\sum_{n\in A}\al_n\be_n\big\|\leq \rho\,\big\|\sum_{n\in A\cup B}\al_n\be_n\big\|,\quad\forall\;\al_n\in\SK.
\label{A1}
\Ee
Observe that \eqref{A1} implies that
\Be
\big\|\sum_{n\in B}\al_n\be_n\big\|\leq (\rho+1)\,\big\|\sum_{n\in A\cup B}\al_n\be_n\big\|,\quad\forall\;\al_n\in\SK.
\label{B1}
\Ee
This condition is clearly satisfied by all Schauder and Ces\`aro bases (with $\rho=K_b$ or $\rho>\beta$), but we shall see below that it also holds in more general situations.

\begin{proposition}\label{Ptheta1}
Let $\{\be_n,\be^*_n\}_{n=1}^\infty$ be an M-basis such that  $\{\be_n\}_{n=1}^\infty$ is $\rho$-admissible. Then
\Be
\cL\geq \frac{\vartheta_m}{(\rho+1)t},\quad\forall\;m\in\SN,\quad t\in(0,1].
\label{thetam}\Ee
\end{proposition}
\begin{proof}
Fix $A\subset\SN$ such that $|A|\leq m$. Choose $C$ disjoint with $A$ such that $|A\cup C|=m$.
Let $n_0=n_0(A\cup C)$ as in the above definition, which we may assume larger than $\max A\cup C$. Pick any $B$ with $\min B\geq n_0$ and $|B|=|A|$, and any $\e,\eta\in\Upsilon$.
Let $x=t\bone_{\e A}+t \bone_C+\bone_{\eta B}$. Then $A\cup C\in G(x,m,t)$, and there is a Chebyshev $t$-greedy operator with $\cht_m(x)$ supported in $A\cup C$. Thus,
\[
\|x-\cht_m(x)\|\leq \cL\,\sigma_m(x)\leq \cL\,\|x-(\bone_{\eta B}+t\bone_C)\|=\cL\,t\,\|\bone_{\e A}\|.
\]
On the other hand, using the property in \eqref{B1} one obtains\[
\|x-\cht_m(x)\|\geq \frac{\|\bone_{\eta B}\|}{\rho+1}.
\]
Thus,
\[
\cL\,\geq\,\frac{1}{(\rho+1)t}\,\frac{\|\bone_{\eta B}\|}{\|\bone_{\e A}\|}.
\]
We now assume additionally that $\min B\geq  n_0+m$, and pick $D\subset[n_0,n_0+m-1]$ such that $|B|+|D|=m$.
Let $y=\bone_{\e A}+t\bone_{\eta B}+t\bone_D$. Then $B\cup D\in G(y,m,t)$ and a similar reasoning gives
\[
\frac{\|\bone_{\e A}\|}\rho\leq \|y-\cht_m(y)\|\leq \cL\,\sigma_m(y)\leq \,\cL\,
t\,\|\bone_{\eta B}\|.
\]
Thus,
\[
\cL\,\geq\,\frac{1}{(\rho+1)t}\,\max\Big\{\frac{\|\bone_{\eta B}\|}{\|\bone_{\e A}\|},\frac{\|\bone_{\e A}\|}{\|\bone_{\eta B}\|}\Big\},
\]
and taking the supremum over all $|B|=|A|$ with $B\geq (n_0+m)A$ and all $\e,\eta\in\Upsilon$, we see that
\[
\cL\,\geq\,\frac{\vartheta_{n_0+m}(A)}{(\rho+1)t}\,\geq \frac{\inf_{c\geq1} \vartheta_{c}(A)}{(\rho+1)t}.
\]
Finally, a supremum over all $|A|\leq m$ leads to \eqref{thetam}.
\end{proof}

\

We now give some general conditions in $\{\be_n,\be^*_n\}_{n=1}^\infty$ and $\SX$ under which $\rho$-admissibility holds. We recall a few standard definitions; see e.g. \cite{Hajek}. We use the notation $[\be_n]_{n\in A}={\overline{\span}}\{\be_n\}_{n\in A}$, for $A\subset\SN$.  A sequence $\{\be_n\}_{n=1}^\infty$ is \emph{weakly null} if\[
\lim_{n\to\infty}x^*(\be_n)=0,\quad\forall\;x^*\in\SX^*.
\]
Given a subset $Y\subset\SX^*$, we shall say that $\{\be_n\}_{n=1}^\infty$ is \emph{$Y$-null} if\[
\lim_{n\to\infty}y(\be_n)=0,\quad\forall\;y\in Y.
\]
Given $\kappa\in (0,1]$, we say that a set $Y\subset \SX^*$ is  $\kappa$-norming    whenever
$$ \sup_{x^* \in Y, \Vert x^* \Vert \leq 1} \vert x^*(x) \vert\,\geq \,\kappa\,\|x\|, \quad\forall\;x\in\SX  .$$

\begin{proposition}
\label{Ptheta2}
Let $\{\be_n,\be^*_n\}_{n=1}^\infty$ be a biorthogonal system in $\SX\times\SX^*$. Suppose that the sequence $\{\tbe_n:=\|\be^*_n\|\,\be_n\}_{n=1}^\infty\subset\SX$ is $Y$-null, for some subset $Y \subset \SX^*$ which is $\kappa$-norming. 
Then $\{\be_n\}_{n=1}^\infty$ is $\rho$-admissible for every $\rho>1/\kappa$. 
\end{proposition}

\begin{proof}
Consider a finite set $A\subset\SN$ with say $\vert A \vert = m$ and denote $$E :=[\be_n]_{n\in A}.$$ Given $\varepsilon>0$, one can find a finite set $S \subset Y \cap \{x^*\in\SX^*\mid\|x^*\|=1\}$  so that
	\begin{equation}\label{e1} \max_{x^* \in S} \vert x^*(e) \vert \geq (1-\varepsilon) \kappa  \Vert e \Vert,\quad\forall\;e\in E . \end{equation}
	Indeed, it suffices to verify the above inequality for $e$ of norm $1$. Pick an $\varepsilon \kappa/2$-net $(z_k)_{k=1}^N$ in the unit sphere of $E$. For any $k$ find a norm one $z_k^* \in Y$ so that $|z_k^*(z_k)| > (1-\varepsilon/2) \kappa$. We claim that $S = \{z_k^* : 1 \leq k \leq N\}$ has the desired properties. To see this, pick a norm one $e \in E$, and find $k$ with $\|e - z_k\| \leq \varepsilon \kappa/2$. Then
	$$
	\max_{x^* \in S} \vert x^*(e) \vert \geq |z_k^*(e)| \geq |z_k^*(z_k)| - \|e-z_k\| \geq (1-\varepsilon/2) \kappa - \varepsilon \kappa/2 = (1-\varepsilon) \kappa .
	$$
Next, since the sequence $\{\|\be^*_n\|\,\be_n\}$ is $Y$-null, for each $\delta>0$  we can find an integer $n_0> \max A$ so that
	$$
	\max_{x^* \in S} \vert x^*(\be_n) \vert\,\|\be^*_n\|\, \leq \frac{\delta\kappa}{m},\quad\forall\;  n\ge n_0 .
	$$
 Pick any $B$ of cardinality $m$ with $\min B\geq n_0$, and let $$G := [\be_n]_{n\in B}.$$ For $f = \sum_{n \in B} \be^*_n(f) \be_n \in G$, we have 
\begin{equation}\label{e2} \max_{x^* \in S} \vert x^*(f) \vert 
\leq \max_{x^*\in S}\,\sum_{n\in B}|x^*(\be_n)|\,\|\be^*_n\|\,\|f\|\,\leq\,
\delta\kappa  \Vert f \Vert.\end{equation} We claim that
	\begin{equation}
	\Vert e+f \Vert \geq \frac{(1-\e-\dt)\kappa}{1+\dt\kappa}\,\|e\|,\, \, \,
	{\textrm{  for any  }} \, e \in E, \, f \in G .
	\label{eq:direct_sum}
	\end{equation}
To show this, we fix $\ga>0$ (to be chosen later), and assume first that $\|f\|\geq(1+\ga)\|e\|$.
Then,
\[
\|e+f\|\geq\|f\|-\|e\|\geq \ga\|e\|.
\]
Next assume that $\|f\|<(1+\ga)\|e\|$, then using (\ref{e1}) and (\ref{e2}) we obtain that
	$$
	\Vert e+f \Vert \geq \max_{x^* \in S} \vert x^* (e+f) \vert \geq (1-\varepsilon)\kappa \Vert e \Vert - \delta\kappa \Vert f \Vert > (1-\varepsilon- \delta(1+\gamma))\kappa \Vert e \Vert .
	$$
We now choose $\ga$ so that $\ga=(1-\varepsilon- \delta(1+\gamma))\kappa$, that is,
\[
\ga=\frac{(1-\varepsilon- \delta)\kappa}{1+\dt\kappa},
\]
which shows the claim in \eqref{eq:direct_sum}.
Now, given $\rho>1/\kappa$, we may pick $\dt=\e$ sufficiently small so that the above number $\ga>1/\rho$.
Then, \eqref{eq:direct_sum} becomes
\[
\Vert e+f \Vert \geq \frac1\rho\,\|e\|,\, \, \,
	{\textrm{  for any  }} \, e \in [e_n]_{n\in A}, \, f \in [e_n]_{n\in B},\quad 
\]
for all $B$ with $\min B\geq n_0$ and $|B|=|A|=m$. Thus, $\{\be_n\}_{n=1}^\infty$ is $\rho$-admissible.
\end{proof}

\

We mention a few cases where the hypotheses in the above proposition can be applied:

\bline (1) When the sequence $\{\tbe_n\}_{n=1}^\infty$ is weakly null, since $Y = \SX^*$ is always $1$-norming.

\sline (2)  When $\sup_{n\geq1}\|\be_n\|\,\|\be^*_n\|<\infty$ and $Y=[\be^*_n]_{n\in\SN}$ is $\kappa$-norming, since the first condition implies that $\{\tbe_n\}_{n=1}^\infty$ is  $Y$-null. In particular, when $\{\be_n\}_{n=1}^\infty$ is a Schauder basis in $\SX$, in which case the above conditions hold with $\kappa=1/K_b$; see \cite[Theorems I.3.1 and I.12.2]{SingerI}.

\sline (3) In every separable Banach space $\SX$, if one picks $\{\be_n,\be^*_n\}_{n=1}^\infty$ to be an $M$-basis with the properties in (2) and $\kappa=1$; see e.g. \cite[Theorem III.8.5]{SingerII} for the existence of such bases.

\sline (4) Let $\SX=C(K)$ where $K$ is a compact Hausdorff set and let $\mu$ be a Radon probability measure in $K$ with $\supp\mu=K$. Then, the natural embedding of $C(K)$ into $L_\infty(\mu)$ is isometric, and therefore  $Y=L_1(\mu)$ is $1$-norming in $\SX$. Let $\{\be_n\}_{n=1}^\infty$ be a complete system in $\SX$ which is orthonormal with respect to $\mu$ and uniformly bounded, that is, $\int_K \be_n {\overline{\be_m}} \, d\mu = \dt_{n,m}$ and $\sup_n\|\be_n\|_\infty<\infty$. Then the sequence $\{\be_n\}_{n=1}^\infty$ is  $L_1(\mu)$-null in $\SX$. Indeed, this follows from case (2), and the fact that $C(K)$ is dense in $L_1(\mu)$.

Examples of such systems in $C(K)$ include the trigonometric system in $C[0,1]$  (in the real or complex case), as well as certain polygonal versions of the Walsh system \cite{Cie68, ropela79,weisz01}, or any reorderings of them (which may cease to be Ces\`aro bases).

\sline (5) As a dual of the previous, if $\SX=L^1(\mu)$ then every system $\{\be_n\}_{n=1}^\infty$ as in (4) is weakly null, and hence case (1) applies.

\sline (6) Recall the definition of the \emph{right fundamental function}:
$ {\mathbf{\varphi}}_r(m) = \sup \lbrace \Vert \bone_A \Vert : \vert A \vert \leq m \rbrace$.
If $\{\be_n\}_{n=1}^\infty$ is such that ${\mathbf{\varphi}}_r(m) = {\mathbf{o}}(m)$, then this system is weakly null. Indeed, first note that also $ {\mathbf{\tphi}}_r(m) = \sup \lbrace \Vert \bone_{\eta A} \Vert : \vert A \vert \leq m,\;|\eta|=1 \rbrace={\mathbf{o}}(m)$.  Assume that  the system is not weakly null. Then there exist a norm one $x^* \in \SX^*$ and $\varepsilon_0 > 0$ so that the set $A=\{ n \in \mathbb N : \vert x^*(\be_n) \vert \geq \varepsilon_0 \} $ is infinite. Pick any $F \subset A$ with $|F| = m$ and let $\eta_n=\sig[ x^*(\be_n)]$; then 
$$
{\mathbf{\tphi}}_r(m) \geq \Vert \bone_{{\overline{\eta}}F} \Vert \geq \vert x^*(  \sum_{n \in F} {\overline{\eta_n}}\be_n ) \vert =\sum_{n\in F}|x^*(\be_n) | \geq m\varepsilon_0 ,
$$
contradicting our assumption.

\

Finally, as a consequence of Propositions \ref{Ptheta1} and \ref{Ptheta2} one obtains
\begin{theorem}\label{Thm_thetam}
Let $\{\be_n,\be^*_n\}_{n=1}^\infty$ be a seminormalized M-basis such that the sequence $\{\be_n\}_{n=1}^\infty$ is $Y$-null for some subset $Y \subset \SX^*$ which is $\kappa$-norming.
Then, if  $\vartheta_m$ is as in \eqref{theta2}, we have
\Be
\cL\geq \frac{\kappa\,\vartheta_m}{(\kappa+1)t},\quad\forall\;m\in\SN,\quad t\in(0,1].
\label{chtheta}\Ee
\end{theorem}

\section{Examples}\label{examples}

The first two examples are variants of those in \cite[\S5.1]{BBG} and \cite[\S8.1]{BBGHO}.

\subsection{Example \ref{summing}: The summing basis.}\label{summing} Let $\mathbb X$ be the closure of the set of all finite sequences $\mathbf{a}=(a_n)_n\in c_{00}$ with the norm
$$\Vert \mathbf{a}\Vert = \sup_m\Big\vert \sum_{n=1}^m a_n\Big\vert.$$

The canonical system $\cB=\lbrace \be_n\rbrace_{n=1}^\infty$ is a Schauder basis in $\SX$ with $K_b=1$ and $\Vert \be_n\Vert= 1$ for all $n$. Also, $\Vert \mathbf{e}_1^*\Vert = 1$, $\Vert \ben\Vert = 2$ if $n\geq 2$, so $\mathfrak{K}=2$ in Theorem \ref{main2}; see \cite[\S5.1]{BBG}. We now show that, for this example of $(\SX,\cB)$, the bound of Theorem \ref{main2} is sharp. As in \cite[\S5.1]{BBG}, we consider the element:
$$x=\Big(\underbrace{\frac{1}{2},\frac{1}{t},\frac{1}{2}},...,\underbrace{\frac{1}{2},\frac{1}{t},\frac{1}{2}}; \frac{1}{2}; \underbrace{-1,1},...,\underbrace{-1,1},0,...\Big),$$
where we have $m$ blocks of $\left(\frac{1}{2},\frac{1}{t},\frac{1}{2}\right)$ and $m$ blocks of $(-1,1)$. Picking $A=\lbrace n : x_n=-1\rbrace$ as a $t$-greedy set of $x$, we see that
\begin{eqnarray*}
\Vert x-\ch_m^t(x)\Vert &=& \min_{a_i, i =1,...,m}\Big\Vert \Big(\frac{1}{2},\frac{1}{t},\frac{1}{2},...,\frac{1}{2},\frac{1}{t},\frac{1}{2}; \frac{1}{2}; a_1, 1, a_2, 1,...,a_m,1,0,...\Big)\Big\Vert\\
&\geq& \Big\Vert \Big(\frac{1}{2},\frac{1}{t},\frac{1}{2},...,\frac{1}{2},\frac{1}{t},\frac{1}{2}; \frac{1}{2}; 0,...\Big)\Big \Vert = m+\frac{m}{t}+\frac{1}{2}.
\end{eqnarray*}
On the other hand,
\begin{eqnarray*}
\sigma_m(x)&\leq& \Big\Vert x-\frac{t+1}{t}(0,1,0,...,0,1,0; 0,...)\Big\Vert\\
&=& \Big\Vert \Big(\frac{1}{2},-1,\frac{1}{2},...,\frac{1}{2},-1,\frac{1}{2}; \frac{1}{2};-1,1,...,-1,1,0...\Big)\Big\Vert=\frac{1}{2}.
\end{eqnarray*}
Hence, $\cL\geq 1+2(1+\frac1{t})m$ and we conclude that $\cL= 1+2(1+\frac1{t})m$  by Theorem \ref{main2}.
As a consequence, observe that in this case $\ch^t_m(x)=0$.

\begin{Remark}
{\rm The above example strengthens \cite[Theorem 2.4]{SY}, where the authors are only able to show that $1+4m\leq \cLL\leq 1+6m$. }
\end{Remark}

\subsection{Example \ref{diff}: the difference basis.}\label{diff} Let $\lbrace \be_n\rbrace_{n=1}^\infty$ be the canonical basis in $\ell^1(\mathbb N)$ and define the elements
$$y_1 = \be_1,\; y_n = \be_n-\be_{n-1},\; n=2,3,...$$
The new system $\mathcal B= \lbrace y_n\rbrace_{n=1}^\infty$ is called the difference basis of $\ell^1$. We recall some basic properties used in \cite[\S 8.1]{BBGHO}. If $(b_n)_n\in c_{00}$ then
$$\Vert \sum_{n=1}^\infty b_n y_n\Vert = \sum_{n=1}^\infty \vert b_n-b_{n+1}\vert.$$
Also, $\mathcal B$ is a monotone basis with $\Vert y_1\Vert = 1$, $\Vert y_n\Vert = 2$ if $n\geq 2$, and $\Vert y_n^*\Vert = 1$ for all $n\geq 1$ (in fact, the dual system corresponds to the summing basis). So, $\mathfrak{K}=2$ and Theorem \ref{main2} gives $\cL \leq 1+2(1+\frac{1}{t})m$ for all $t\in (0,1]$. To show the equality we consider the vector $x=\sum_nb_n y_n$ with coefficients $(b_n)$ given by
$$\Big(1,\underbrace{1,1,-\tfrac{1}{t},1},...,\underbrace{1,1,-\tfrac{1}{t},1},0,...\Big),$$
where the block $\Big(1,1,\frac{-1}{t},1\Big)$ is repeated $m$ times.
If we take $\Gamma =\lbrace 2,6,...,4m-2\rbrace$ as a $t$-greedy set for $x$ of cardinality $m$, then
\begin{eqnarray*}
\Vert x-\ch_m^t(x)\Vert&=&\inf_{(a_j)_{j=1}^m}\Vert x-\sum_{j=1}^m a_jy_{4j-2}\Vert\\
&=& \inf_{(a_j)_{j=1}^m}\Big\Vert \Big(1,1-a_1,1,\frac{-1}{t},1,...,1-a_m,1,\frac{-1}{t},1,0,...\Big)\Big\Vert\\
&=& \inf_{(a_j)_{j=1}^m}2\sum_{j=1}^m \vert a_j\vert + 2m\Big(1+\frac{1}{t}\Big)+1=2m\Big(1+\frac{1}{t}\Big)+1.
\end{eqnarray*}
Hence, in this case we also have $\ch^t_m(x)=0$. On the other hand
\[
\sigma_m(x)\leq \big\|x+(1+\tfrac1t)\sum_{j=1}^m y_{4j}\big\|= \Vert (1,1,1,1,1,...,1,1,1,1,0,...)\Vert =1.\]
This shows that $\cL = 1+2(1+\frac{1}{t})m$.

\subsection{Example \ref{ex_trig}: the trigonometric system in $L^p(\mathbb T)$.}\label{ex_trig} Consider $\mathcal B=\{e^{i nx}\}_{n\in\SZ}$ in $L^p(\ST)$ for $1\leq p<\infty$, and in $C(\ST)$ if $p=\infty$.
In \cite{Tem98trig}, Temlyakov  showed that \[
c_pm^{|\frac1p-\frac12|}\leq \bL\leq 1+3m^{|\frac1p-\frac12|},
\]
for some $c_p>0$ and all $1\leq p\leq \infty$.
Adapting his argument, Shao and Ye have recently established, in \cite[Theorem 2.1]{SY}, that for $1<p\leq\infty$ it also holds \Be\cLL\approx m^{|\frac1p-\frac12|}.\label{LcT}\Ee
The case $p=1$ is left as an open question, and only the estimate  $\frac{\sqrt m}{\ln(m)}\lesssim \cLL\lesssim \sqrt m$ is given; see \cite[(2.24)]{SY}.
Moreover, the proof of the case $p=\infty$ seems to contain some gaps and may not be complete.

Here, we shall give a short proof ensuring the validity of \eqref{LcT} in the full range $1\leq p\leq\infty$,
with a reasoning similar to \cite[\S 5.4]{BBG}.
More precisely, we shall prove the following.

\begin{proposition}
Let $1\leq p\leq \infty$. Then there exists $c_p>0$ such that
\Be
\cL\,\geq \, c_p\,t^{-1}\, m^{|\frac 1p-\frac 12|},\quad \forall\;m\in\SN,\quad t\in(0,1].
\label{Tt}\Ee
\end{proposition}
We remark that in the cases $p=1$ and $p=\infty$ the trigonometric system is not a Schauder basis, but it is a Ces\`aro basis\footnote{We equip $\cB$ with its natural ordering $\{1,e^{ix}, e^{-ix}, e^{2ix}, e^{-2ix}, \ldots\}$.}.
So we may use the lower bounds in Theorem \ref{th3}, namely
\Be
\cL\geq \,{c'_p}\;t^{-1}\,\sup_{{{|A|=|B|\leq m}\atop{A>2B\;\;\text{or}\;\;B>2A}}}\,\sup_{|\e|=|\eta|=1}\;\frac{\|\bone_{\e A}\|}{\|\bone_{\eta B}\|}.
\label{prop1main}
\Ee

\begin{itemize}
\item Case $1<p\leq 2$. Assume that $m=2\ell+1$ or $2\ell+2$ (that is, $\ell=\lfloor \frac{m-1}2\rfloor$). We choose  $B=\{-\ell,...,\ell\}$, so that $\bone_B=D_\ell$ is the $\ell$-th Dirichlet kernel, and hence
\[
\|\bone_B\|_p=\|D_\ell\|_{L^p(\ST)}\approx m^{1-\frac1p}.
\] Next we take a lacunary set $A=\lbrace 2^j : j_0\leq j\leq j_0+2\ell\rbrace$, so that
\Be
\|\bone_A\|_p\approx\sqrt m,
\label{Ap}\Ee
and where $j_0$ is chosen such that $2^{j_0}\geq m$, and hence $A>2B$. Then, \eqref{prop1main}
implies \[
\cL\geq \,{c_p}\;t^{-1}\frac{m^{1/2}}{m^{1-\frac1p}}\,=\,c_p\,t^{-1}\, m^{|\frac 1p-\frac 12|}.
\]

\item Case $2\leq p<\infty$. The same proof works in this case, just reversing the roles of $A$ and $B$.

\item Case $p=\infty$. We replace the lacunary set by a Rudin-Shapiro polynomial of the form
\[
R(x)=e^{iNx}\,\sum_{n=0}^{2^L-1}\e_n e^{inx},\quad \mbox{with }\;\e_n\in\{\pm1\},
\]
where $L$ is such that $2^L\leq m<2^{L+1}$; see e.g. \cite[p. 33]{katz}. Then, $R=\bone_{\e B}$ with $B=N+\{0,1,\ldots,2^L-1\}$ and
\[
\|\bone_{\e B}\|_\infty=\|R\|_{L^\infty(\ST)}\approx \sqrt m.
\]
If we pick $N\geq 2\cdot 2^L$, then $B>2A$ with $A=\{\pm1,\ldots,\pm (2^L-1)\}$. Finally,
\[
\|\bone_A\|_\infty=\|D_{2^L-1}-1\|_{L^\infty(\ST)} \approx\, m.
\]
So, \eqref{prop1main}
implies the desired bound.

\item Case $p=1$. We use the lower bound in Lemma \ref{lem_x}, namely\Be
\cL \,\geq\, c_1'\;t^{-1}\;\frac{\|\bone_A\|}{\|\bone_B+y\|},
\label{cL1}
\Ee
for all $|A|=|B|\leq m$ and all $y$ such that $A>2(B\cupdot\supp y)$
and $\sup_n|\be^*_n(y)|\leq 1$.
As before, let $m=2\ell+1$ or $2\ell+2$, and choose the same sets $A$ and $B$ as in the case $1<p\leq2$.
Next choose $y$ so that the vector
\[
V_\ell=\bone_B+y
\]
is a de la Vall\'ee-Poussin kernel as in \cite[p. 15]{katz}.
Then, the Fourier coeffients $\be^*_n(y)$ have modulus $\leq1$ and are supported in $\{n\mid \ell<|n|\leq 2\ell+1\}$,
so the condition  $A>2(B\cupdot\supp y)$ holds if $2^{j_0}\geq 2m+1$.
Finally,
\[
\|\bone_B+y\|_1=\|V_\ell\|_{L^1(\ST)}\leq 3,
\]
so the bound $\cL\gtrsim t^{-1}\sqrt m$ follows from \eqref{cL1}.
\end{itemize}

\begin{Remark}
{\rm Using the trivial upper bound $\cL\leq \bL^t\lesssim t^{-1} m^{|\frac1p-\frac12|}$, we conclude that $\cL \approx t^{-1}m^{\vert \frac{1}{p}-\frac{1}{2}\vert}$ for all $1\leq p\leq \infty$.}
\end{Remark}


\section{Comparison between $\tmu_m$ and $\tmu^d_m$ }\label{s:comparison_mu_mu_d}
\setcounter{equation}{0}\setcounter{footnote}{0}
\setcounter{figure}{0}

In this section we compare the democracy constants $\tmu_m$ and $\tmu^d_m$ defined in \S1 above. Let us first note that
\Be
\tmu^d_m\leq \tmu_m\leq (\tmu^d_m)^2
\label{mu2}
\Ee
and
\Be
\tmu_m^d\leq \tmu_m\leq (1+2\kappa)\ga_m\tmu_m^d,\label{mumud}
\Ee
where $\kappa=1$ or 2 depending if $\SK=\SR$ or $\SC$.
Indeed, the left inequality in  (\ref{mu2}) is immediate by definition, and the right one  follows from
\[
\frac{\|\bone_{\eta B}\|}{\|\bone_{\e A}\|}= \frac{\|\bone_{\eta B}\|}{\|\bone_{C}\|}\,\frac{\|\bone_{C}\|}{\|\bone_{\e A}\|}\leq (\tmu^d_m)^2,
\]
for any $|A|=|B|\leq m$ and any $C$ disjoint with $A\cup B$ with $|C|=|A|=|B|$.
Concerning the right inequality in (\ref{mumud}), we use that if $|A|=|B|\leq m$ then
\[
\frac{\|\bone_{\e A}\|}{\|\bone_{\eta B}\|}\leq
\frac{\|\bone_{\e (A\setminus B)}\|+\|\bone_{\e(A\cap B)}\|}{\|\bone_{\eta B}\|}\leq \ga_m \frac{\|\bone_{\e (A\setminus B)}\|}{\|\bone_{\eta (B\setminus A)}\|}
+\frac{\|\bone_{\e(A\cap B)}\|}{\|\bone_{\eta B}\|}\leq \ga_m\,\tmu_m^d+2\kappa\ga_m,
\]
using in the last step \cite[Lemma 3.3]{BBG}.
From (\ref{mumud}) we see that  $\tmu_m\approx\tmu_m^d$ when $\cB$ is quasi-greedy for constant coefficients.

In the next subsection we shall show that $\tmu_m\approx \tmu^d_m$ for all Schauder
bases, a result  which seems new in the literature. 

\subsection{Equivalence for Schauder bases}

We begin with a simple observation.
\begin{lemma}\label{s:alternative_mu}
\Be
\tmu^d_m=\sup\Big\{\frac{\|\bone_{\eta B}\|}{\|\bone_{\e A}\|}\mid |B|\leq |A|\leq m,\;\;A\cap B=\emptyset,\;\;|\e|=|\eta|=1\Big\}.
\label{tmudN1}\Ee
\end{lemma}
\begin{proof}
Let $|\e|=|\eta|=1$ and $|B| \leq |A| \leq m$ with $A\cap B=\emptyset$. We must show that
$\|\bone_{\eta B}\|/\|\bone_{\e A}\|\leq \tmu^d_m$. Pick any set $C$ disjoint with $A\cup B$  such that
$|B|+|C| = |A|$.
We now use the elementary inequality
\Be
\|x\|=\Big\|\frac{x+y}2+\frac{x-y}2\Big\|\leq\max\{\|x+y\|,\|x-y\|\},\label{max_xy}
\Ee
with $x=\bone_{\eta B}$ and $y=\bone_C$. Let $\eta'\in\Upsilon$ be such that $\eta'|_B = \eta|_B$ and $\eta'|_C=\pm1$, according to the sign that reaches the maximum in \eqref{max_xy}.
Then $\|\bone_{\eta B}\|\leq \|\bone_{\eta'(B\cup C)}\|\leq \tmu^d_m\|\bone_{\e A}\|$, and the result follows.
\end{proof}

\begin{theorem}\label{th_mudN}
If $K_b$ is the basis constant and  $\varkappa=\sup_{n}\Vert \ben\Vert\Vert \be_n\Vert$, then
\Be
\tmu_m\leq 2(K_b+1)\tmu^d_m +\varkappa\,K_b.\label{mudN}\Ee
\end{theorem}
\begin{proof}
Let $|A|=|B|\leq m$, and $|\e|=|\eta|=1$. Then
\[
\frac{\|\bone_{\eta B}\|}{\|\bone_{\e A}\|} \leq \frac{\|\bone_{\eta(B\setminus A)}\|}{\|\bone_{\e A}\|}+\frac{\|\bone_{\eta (B\cap A)}\|}{\|\bone_{\e A}\|} = I+II.
\]
Lemma \ref{s:alternative_mu} implies $I\leq \tmu_m^d$.
We now bound $II$. Pick an integer $n_0$ such that $A_1=\{n\in A\mid n\leq n_0\}$ and $A_2=A\setminus A_1$ satisfy
\[
|A_1|=|A_2| \quad\mbox{(if $|A|$ is even), \; or}\quad |A_1|=\frac{|A|-1}2=|A_2|-1\quad \mbox{(if $|A|$ is odd)}.\]
Then
\Beas
II & \leq & \frac{\|\bone_{\eta (B\cap A_1)}\|}{\|\bone_{\e A}\|}+\frac{\|\bone_{\eta (B\cap A_2)}\|}{\|\bone_{\e A}\|}\\
& \leq & (K_b+1)\frac{\|\bone_{\eta (B\cap A_1)}\|}{\|\bone_{\e A_2}\|}+K_b\frac{\|\bone_{\eta (B\cap A_2)}\|}{\|\bone_{\e A_1}\|}\;=\;II_1+II_2,
\Eeas
using in the second line the basis constant bound for the denominator. Since $|B\cap A_1|\leq |A_1|\leq |A_2|$, we see that \[
II_1\leq (K_b+1)\tmu^d_m.\]
On the other hand, picking any number $n_1\in B\cap A_2$, and using $\|\be^*_{n_1}\|\|\bone_{\e A}\|\geq |\be^*_{n_1}(\bone_{\e A})|=1$, we see that
\[
II_2\leq K_b\frac{\|\bone_{\eta (B\cap A_2\setminus\{n_1\})}\|}{\|\bone_{\e A_1}\|}+
K_b\|\be_{n_1}\|\|\be^*_{n_1}\|\leq K_b\tmu^d_m+\varkappa K_b,
\]
the last bound due to $|B\cap A_2\setminus\{n_1\}|\leq |A_2|-1\leq |A_1|$ and Lemma \ref{s:alternative_mu}.
Putting together the previous bounds easily leads to \eqref{mudN}.
\end{proof}

\begin{Remark}{\rm
A similar argument shows the equivalence of the standard (unsigned) democracy parameters
\Be
\mu_m=\sup_{|A|=|B|\leq m}\frac{\|\bone_{B}\|}{\|\bone_{A}\|}\mand
\mu^d_m=\sup_{{|A|=|B|\leq m}\atop{A\cap B=\emptyset}}\frac{\|\bone_{B}\|}{\|\bone_{A}\|}.
\label{unsign}\Ee
Indeed, in this case, the analog of \eqref{tmudN1} takes the weaker form
\Be
\mu^d_m\leq \sup_{{|B|\leq |A|\leq m}\atop{A\cap B=\emptyset}}\frac{\|\bone_{ B}\|}{\|\bone_{ A}\|}\leq K_b\mu^d_m.
\label{mudN1}\Ee
Then, \eqref{mudN1} and the same proof we gave for Theorem \ref{th_mudN} (with $\eta=\e\equiv1$) leads to
\Be
\mu_m\leq 2(K_b+1)K_b\,\mu^d_m +\varkappa\,K_b.
\label{mudN2}
\Ee
 }
\end{Remark}

\

\subsection{An example where $\tmu_m$ grows faster than $\tmu^d_m$}

The following example also seems to be new in the literature. As in \eqref{unsign}, we denote by $\mu_m$, $\mu^d_m$ the democracy parameters corresponding to constant signs.

\begin{theorem}\label{Timur}
	There exists a Banach space $\SX$ with an M-basis $\cB$ such that\[
	\limsup_{m\to\infty}\frac{\tmu_m}{[\tmu^d_m]^{2-\e}}=\limsup_{m\to\infty}\frac{\mu_m}{[\mu^d_m]^{2-\e}}=\infty,\quad \forall\;\e>0.
	\]
\end{theorem}
\begin{proof}
Let $N_0=1$, and define recursively $N_k=2^{2^{N_{k-1}}}$, and $N_k'=N_1+\ldots +N_{k-1}$.
Consider the blocks of integers
\[
S_k=\big\{N'_k+1,\ldots, N'_k+N_k
\big\},\]
and denote the tail blocks by $T_k=\cup_{j\geq k+1} S_{j}$. Finally, let
\[
\fN_k=\big\{(\sigma_j)_{j\in S_k}\mid\sigma_j\in\{\pm1\}\mand \sum_{j\in S_k}\sigma_j=0\big\}.
\]
We define a real Banach space $\SX$ as the closure of $c_{00}$ with the norm
\[
\|x\|\;=\;\max\Big\{\;\|x\|_\infty,
\;\sup_{k\geq1} \al_k\,\sup_{\sigma\in \fN_k}\big|\langle \bone_{\sigma S_k},x\rangle\big|,\;
\sup_{k\geq1}\beta_k\sup_{{S\subset T_k}\atop{|S|=N_k}}\sum_{j\in S}|x_j|\;\Big\},
\]
where the weights $\al_k$ and $\beta_k$ are chosen as follows:
\[
\al_k=2^{-N_{k-1}}=\frac1{\log_2 N_k}\mand \beta_k=\frac1{\sqrt{N_k}}.
\]
Observe that
\[
N'_k=N_1+\ldots+N_{k-1}\leq 2 N_{k-1}=2\log_2\log_2 N_k
\mand \frac{\al_k}{\beta_k}=\frac{\sqrt{N_k}}{\log_2 N_k}.
\]
{\bf Claim 1:} $\quad\Ds\tmu_{N_k}\geq \mu_{N_k}\geq \frac{N_k/2}{(\log_2 N_k)\sqrt{\log_2\log_2 N_k}}$, for all $k\geq1$.

\begin{proof}
Pick any $A\subset S_k\cup S_{k+1}$ such that $|A|=N_k$ and $|A\cap S_k|=|A\cap S_{k+1}|=N_k/2$. Then
\[
\|\bone_A\|\geq \al_k\,N_k/2=\frac{N_k/2}{\log_2 N_k}.
\]
Next, pick $B=S_k$, so that $|B|=|A|=N_k$ and
\[
\|\bone_B\|=\max\Big\{\;1,\;\al_k\cdot 0,\;\sup_{n\leq k-1}\beta_n N_n\;\Big\}=\beta_{k-1}N_{k-1}=\sqrt{N_{k-1}}=
\sqrt{\log_2\log_2 N_k}.
\]
Then $\mu_{N_k}\geq \|\bone_A\|/\|\bone_B\|\geq \frac{N_k/2}{(\log_2 N_k)\sqrt{\log_2\log_2 N_k}}$.
\end{proof}

\bline {\bf Claim 2:} $\quad\Ds\mu^d_{N_k}\leq \tmu^d_{N_k}\leq \sqrt{N_k}$, for all $k\geq2$.

\begin{proof}
Let $A,B$ be any pair of disjoint sets with $|A|=|B|\leq N_k$, and let $|\e|=|\eta|=1$.
If $|A|=|B|\leq \sqrt{N_k}$, then the trivial bounds $\|\bone_{\e A}\|\leq |A|$ and $\|\bone_{\eta B}\|\geq1$ give
\[
\frac{\|\bone_{\e A}\|}{\|\bone_{\eta B}\|}\leq \sqrt{N_k}.
\]
So, it remains to consider the cases $\sqrt{N_k}<|A|=|B|\leq N_k$. We split $A$ into three parts\[
A_0=A\cap S_k,\quad A_+=A\cap T_k,\quad A_-=A\cap[S_1\cup\ldots\cup S_{k-1}].
\]
Then, we have the following upper bound
\Beas
\|\bone_{\e A}\| & \leq&  \max\Big\{1,\; \sup_{n<k}\al_{n}|A_-|, \;\al_k|A_0|,\;\sup_{n>k}\al_n N_k,\;
\sup_{n<k}\beta_nN_n,\;\sup_{n\geq k}\beta_n|A|\;\Big\}\\
& \leq & \max\Big\{\;  N'_k, \;\al_k|A_0|,\;\beta_k|A|\;\Big\},
\Eeas
due to the elementary inequalities
\Bi
\item $\sup_{n<k} \al_n|A_-|\leq |A_-|\leq N'_k
$
\item $\sup_{n>k}\al_n N_k=\al_{k+1} N_k=N_k2^{-N_k}\leq1$
\item $\sup_{n<k}\beta_nN_n=\sqrt{N_{k-1}}\leq N_{k-1}\leq N'_k$
\item $\sup_{n\geq k}\beta_n|A|=\beta_k|A|$.
\Ei
Moreover, since $\beta_k|A|\leq\min\{\beta_kN_k=\sqrt{N_k},\;\al_k|A|\;\}$, we derive
\Be
\|\bone_{\e A}\|\leq \max \{\sqrt{N_k}, \al_k|A_0|\}\mand \|\bone_{\e A}\|\leq \max \{N'_k,\al_k|A|\}.
\label{upperA}\Ee
We now give a lower bound for $\|\bone_{\eta B}\|$. The key estimate will rely on the following

\begin{lemma}
Let $B_0=B\cap S_k$ and $B^c_0=S_k\setminus B_0$. Then
\Be
\sup_{\sigma\in \fN_k}\big|\langle\bone_{\sigma S_k},\bone_{\eta B_0}\rangle\big|\,\geq\,
\min\{|B_0|,|B_0^c|\}.
\label{Timur0}\Ee
\end{lemma}
\begin{proof}
If $|B_0| \leq N_k/2$, then we may select any $\sigma \in \fN_k$ such that $\sigma|_{B_0} = \eta$ (which is possible since $\vert B_0^c \vert \geq \vert B_0 \vert$), which gives
	$$
	 |\langle \bone_{\sigma S_k}, \bone_{\eta B_0} \rangle | = |B_0|=\min\{|B_0|,|B_0^c|\}.
	$$
Assume now that $\vert B_0 \vert > N_k/2$. Pick any $S \subset B_0$ with $\vert S \vert = \vert B_0^c \vert = N_k - \vert B_0 \vert$. Choose $\nu \in \lbrace -1,1\rbrace^{B_0^c}$ so that $\sum_{i \in S} \eta_i + \sum_{i \in B_0^c} \nu_i = 0$. Choose $\tau \in \lbrace -1,1\rbrace^{B_0 \backslash S}$ so that $\sum_{i \in B_0 \backslash S} \tau_i = 0$. Replacing $\tau$ by $-\tau$, if necessary, we may assume that $\sum_{i\in B_0\setminus S}\tau_i\eta_i\geq0$. Finally, define $\sigma\in \fN_k$ by setting
\[
\sigma|_S= \eta|_S, \quad \sigma|_{B^c_0}= \nu|_{B^c_0}, \quad \sigma|_{B_0 \backslash S}=\tau|_{B_0 \backslash S}.
\]
 Then,
	$$
	 |\langle \bone_{\sigma S_k}, \bone_{\eta B_0} \rangle |\, =\, \sum_{i\in S} \eta_i^2 + \sum_{i \in B_0 \backslash S} \tau_i \eta_i \, \geq  
	|S| = |B_0^c|  =\min\{|B_0|,|B_0^c|\}   \,. \qedhere
	$$
\end{proof}
From the lemma and the definition of the norm we see that
\Be
\|\bone_{\eta B}\|\geq \max\Big\{\,1,\; \al_k\min\{|B_0|,|B^c_0|\},\;\beta_k|B_+|\;\Big\}.
\label{lowB}\Ee
We shall finally combine the estimates in \eqref{upperA} and \eqref{lowB} to establish Claim 2.
We distinguish two cases

\sline\emph{Case 1:} $\min\{|B_0|,|B^c_0|\}=|B^c_0|$. Then, since $A_0\subset B^c_0$, we see that
\[
\al_k|A_0|\leq \al_k|B^c_0|\leq \|\bone_{\eta B}\|,
\]
and therefore the first estimate in \eqref{upperA} gives\[
\frac{\|\bone_{\e A}\|}{\|\bone_{\eta B}\|}\leq\frac{\max\{\sqrt{N_k},\|\bone_{\eta B}\|\}}{\|\bone_{\eta B}\|}
\leq \sqrt{N_k}.
\]
\sline\emph{Case 2:} $\min\{|B_0|,|B^c_0|\}=|B_0|$. Then, \eqref{lowB} reduces to
\[
\|\bone_{\eta B}\|  \geq    \max\big\{\,\al_k|B_0|,\;\beta_k|B_+|\;\big\}\geq\beta_k\frac{|B_0|+|B_+|}2=  \beta_k\frac{|B|-|B_-|}2\geq \beta_k|B|/4,\]
since $|B_-|\leq N'_k\leq\sqrt{N_k}/2\leq |B|/2$, if $k\geq2$. Also, the second bound in \eqref{upperA} reads
\[
\|\bone_{\e A}\|\leq \al_k |A|,
\]
since $N'_k\leq \sqrt{N_k}/\log_2 N_k=\al_k\sqrt{N_k}\leq \al_k|A|$, if $k\geq2$. Thus
\[
\frac{\|\bone_{\e A}\|}{\|\bone_{\eta B}\|}\leq\frac{\al_k|A|}{\beta_k|B|/4}=\frac{4\al_k}{\beta_k}=
\frac{4\sqrt{N_k}}{\log_2 N_k}
\leq \sqrt{N_k}.
\]
This establishes Claim 2.
\end{proof}
From Claims 1 and 2 we now deduce that
\[
\frac{\mu_{N_k}}{[\tmu^d_{N_k}]^{2-\e}}\geq \frac{N_k^{\e/2}/2}{(\log_2 N_k)\sqrt{\log_2\log_2 N_k}}\to\infty,
\]
and therefore
\[
\limsup_{N\to\infty}\frac{\mu_{N}}{[\mu^d_{N}]^{2-\e}}=\limsup_{N\to\infty}\frac{\tmu_{N}}{[\tmu^d_{N}]^{2-\e}}=\infty. \qedhere
\]
\end{proof}

\section{Norm convergence of $\cht_m x$ and $\cG_m^t x$}

In this section we search for conditions in $\cB=\{\be_n\}_{n=1}^\infty$ under which it holds
\Be
\Vert x-\ch_m(x)\Vert\to0,\quad \forall\;x\in\SX.\label{SYconv}\Ee
In \cite[Theorem 1.1]{SY} this convergence is asserted for all 
``bases'' $\{\be_n,\be^*_n\}_{n=1}^\infty$ satisfying (a)-(b)-(c).
The proof however, does not seem complete, so we investigate here whether \eqref{SYconv}
may be true in that generality.

\

The solution to this question requires the notion of \emph{strong M-basis}; see \cite[Def 8.4]{SingerII}.
We say that $\cB$ is a strong M-basis if additionally to the conditions (a)-(d) in \S1 it also holds
\Be
\,{\overline{\span\{\be_n\}_{n\in A}}}\,=\,\big\{x\in\SX\mid \supp x\subset A\big\},\quad \forall\;A\subset\SN.
\label{strongM}\Ee
Clearly, all Schauder or Ces\`aro bases (in some ordering) are strong M-bases; see e.g. \cite{ruckle} for further examples. However, there exist M-bases which are not strong M-bases, see e.g. \cite[p. 244]{SingerII}, or \cite{dov}\footnote{We thank V. Kadets for kindly providing this reference.} for seminormalized examples in Hilbert spaces.

\begin{lemma} If $\cB$ is an M-basis which is not a strong M-basis, then there exists an $x_0\in\SX$ such that,
	for all Chebyshev greedy operators $\ch_m$,
	\Be
	\liminf_{m\to\infty}\|x_0-\ch_mx_0\|>0.\label{liminf}
	\Ee
\end{lemma}
\begin{proof}
	If $\cB$ is not a strong M-basis there exists some set $A\subset\SN$ (necessarily infinite) and some $x_0\in\SX$ with $\supp x_0\subset A$ such that
	\[
	\dt=\dist (x_0,[\be_n]_A)>0.\]
	Since $\supp\ch_mx_0$ is always a subset of $A$, this implies \eqref{liminf}.
\end{proof}
\begin{Remark}{\rm
		The above reasoning also implies that $\liminf_{m}\|x_0-\cG_mx_0\|>0$, for all greedy operators $\cG_m$.
		In particular, for M-bases which are not strong, the quasi-greedy condition \Be\label{Cq}
		C_q:=\sup_{{\cG_m\in\SG_m}\atop{m\in\SN}}\|\cG_m\|<\infty\Ee
		does not imply that $\cG_mx$ converges to $x$ for all $x\in\SX$.
		So the standard characterization in \cite[Theorem 1]{Wo} needs the extra assumption that $\cB$ is a strong M-basis.
}\end{Remark}

A corrected version of \cite[Theorem 1.1]{SY} (and also of ``$3\Rightarrow1$'' in \cite[Theorem 1]{Wo}) is the following.

\begin{proposition}
	If $\cB$ is a strong M-basis then, for all Chebyshev $t$-greedy operators $\ch^t_m$, 
	\Be
	\lim_{m\to\infty}\|x-\ch^t_mx\|=0,\quad\forall\;x\in\SX. 
	\label{limch}
	\Ee
	If additionally $C_q<\infty$, then for all $t$-greedy operators $\cG^t_m$, 
	\Be
	\lim_{m\to\infty}\|x-\cG^t_mx\|=0,\quad\forall\;x\in\SX. 
	\label{limcG}
	\Ee
\end{proposition}
\begin{proof}
	Given $x\in\SX$ and $\e>0$, by \eqref{strongM} there exists $z=\sum_{n\in B}b_n\be_n$ such that $\|x-z\|<\e$,
	for some finite set $B\subset\supp x$. Let $\al=\min_{n\in B}|\be^*_n(x)|$ and 
	\[\bLa_{\al}=\{n\mid |\be^*_n(x)|\geq \al\}.\]
	Since $\al>0$, this is a finite greedy set for $x$ which contains $B$. Moreover, we claim that
	\Be
	\bLa_{\al}\subset \supp\ch^t_m x=:A,\quad \forall\; m> |\bLa_{t\al}|.
	\label{incl}\Ee
	Indeed, if this was not the case there would exist $n_0\in\bLa_{\al}\setminus A$, and since $A$ is a $t$-greedy set for $x$, then $\min_{n\in A}|\ben(x)|\geq t|\be^*_{n_0}(x)|\geq t\al$. So, $A\subset\bLa_{t\al}$, which is a contradiction since $m=|A|>|\bLa_{t\al}|$.
	Therefore, \eqref{incl} holds and hence
	\[
	\|x-\ch^t_mx\|\leq \|x-\sum_{n\in B}b_n\be_n\|<\e,\quad \forall\; m> |\bLa_{t\al}|.
	\]
	This establishes \eqref{limch}.
	
	\
	
	We now prove \eqref{limcG}. As above, let $z=\sum_{n\in B}b_n\be_n$ with $B\subset\supp x$ and $\|x-z\|<\e$.
	Performing if necessary a small perturbation in the $b_n$'s, we may assume that $b_n\not=\ben(x)$ for all $n\in B$.
	Let now \[
	\al_1=\min_{n\in B}|\be^*_n(x)|,\quad \al_2=\min_{n\in B}|\be^*_n(x-z)|,\mand \al=\min\{\al_1,\al_2\}>0.\]
	Consider the sets
	\[\bLa_{t\al}=\{n\mid |\be^*_n(x)|\geq t\al\}=\{n\mid |\be^*_n(x-z)|\geq t\al\},\]
	which for all $t\in(0,1]$ are greedy sets for \emph{both} $x$ and $x-z$, and contain $B$.
	We claim that, \Be
	\mbox{if $m>|\bLa_{t\al}|$ and $A:=\supp\cG^t_mx$,} \quad {\rm then }\quad
	\bLa_{\al}\subset A \mand A\in G(x-z,m,t).
	\label{incl2}\Ee
	The assertion $\bLa_{\al}\subset A$ is proved exactly as in \eqref{incl}.
	Next, we must show that
	\[
	\mbox{if $n\in A$} \quad {\rm then }\quad |\be^*_n(x-z)|\;\geq\; t\,\max_{k\notin A}|\be^*_k(x-z)|\,=\; t\,\max_{k\notin A}|\be^*_k(x)|.
	\]
	This is clear if $n\in A\setminus B$ since $\ben(x-z)=\ben(x)$, and $A\in G(x,m,t)$.
	On the other hand, if $n\in B$, then $|\ben(x-z)|\geq \al_2\geq \al\geq \max_{k\in A^c}|\be^*_k(x)|$,
	the last inequality due to $\bLa_\al\subset A$.
	Thus \eqref{incl2} holds true, and therefore
	\[
	\cG^t_m(x)-z=\sum_{n\in A}\ben(x-z)\be_n\,=\,\bar{\cG}^t_m(x-z), 
	\]
	for some $\bar{\cG}^t_m\in\SG^t_m$. Thus,
	\[
	\|\cG^t_m(x)-x\| 
	\,=\, \|(I-\bar{\cG}^t_m)(x-z)\| 
	\,\leq \, (1+\|\bar{\cG}^t_m\|)\,\e,\]
	and the result follows from $\sup_{m}\|\bar{\cG}^t_m\|\leq (1+4C_q/t)C_q$, by \cite[Lemma 2.1]{DKO}.

\end{proof}

\bibliographystyle{plain}

\begin{thebibliography}{1}


\bibitem{AA1} \textsc{F. Albiac and J.L. Ansorena},
\emph{Characterization of 1-almost greedy bases}.
Rev. Matem. Compl. {\bf 30} (1) (2017), 13--24.


%
%

\bibitem{B} \textsc{P. M. Bern\'a},
\emph{Equivalence between almost-greedy and semi-greedy bases}.
 J. Math. Anal. Appl. {\bf 417} (2019), 218--225.

\bibitem{BB1} \textsc{P. M. Bern\'a, \'O. Blasco},
\emph{Characterization of greedy bases in Banach spaces}.
J. Approx. Theory {\bf 215} (2017), 28--39.

\bibitem{BBG} \textsc{P. M. Bern\'a, \'O. Blasco, G. Garrig\'os},
\emph{Lebesgue inequalities for the greedy algorithm in general bases}.
Rev. Mat. Complut. \textbf{30} (2017), 369--392.

\bibitem{BBGHO} \textsc{P. M. Bern\'a, \'O. Blasco, G. Garrig\'os, E. Hern\'andez, T. Oikhberg},
\emph{Embeddings and Lebesgue-Type Inequalities for the Greedy Algorithm in Banach Spaces}.
Constr. Approx. {\bf 48} (3) (2018), 415--451.


\bibitem{Cie68}
\textsc{Z. Ciesielski}, A bounded orthonormal system of polygonals. Studia Math. {\bf 31}  (1968) 339--346.



\bibitem{DKK}
\textsc{S.J. Dilworth, N.J. Kalton, D. Kutzarova},
\emph{On the existence of almost greedy bases in Banach spaces}, Studia Math. 159 (1) (2003),  67--101.

\bibitem{DKKT}
\textsc{S.J. Dilworth, N.J. Kalton, D. Kutzarova, and V.N.
Temlyakov}, \emph{The Thresholding Greedy Algorithm, Greedy Bases,
and Duality}, Constr. Approx. 19, (2003), 575--597.


\bibitem{DKO}
\textsc{S.J. Dilworth, D. Kutzarova, T. Oikhberg}, \emph{Lebesgue constants for the weak greedy algorithm},
Rev. Matem. Compl. 28 (2) (2015), 393--409.



\bibitem{dov}
\textsc{L. N. Dovbysh, N. K. Nikolskii,
	V. N. Sudakov}, \emph{How good can a nonhereditary family be?}
J. Sov. Math. {\bf 34} (6) (1986), 2050--2060.

%



%

\bibitem{GHO}
\textsc{G. Garrig\'os, E. Hern\'andez and T. Oikhberg}, \emph{Lebesgue-type inequalities for quasi-greedy
bases}, Constr. Approx. 38 (3) (2013), 447--470.

%
%
\bibitem{Hajek}
\textsc{P. Hajek, V. Montesinos Santaluc\'\i a, J. Vanderwerff, V. Zizler},
\emph{ Biorthogonal systems in Banach spaces}.
Springer-Verlag 2008.

%
%
%


\bibitem{katz}
\textsc{Y. Katznelson}, \emph{An introduction to Harmonic Analysis}, 2nd ed.  Dover Publ Inc, New York, 1976.

%
%
\bibitem{KT}
\textsc{S.V. Konyagin and V.N. Temlyakov}, \emph{A remark on greedy
approximation in Banach spaces}, East. J. Approx. 5 (1999),
365--379.

\bibitem{KT2}
\textsc{S.V. Konyagin and V.N. Temlyakov}, \emph{Greedy approximatin with regard to bases and general minimal systems}, Serdica Math. J. \textbf{28} (2002), 305--328.
%

%


%




\bibitem{ropela79} \textsc{S. Ropela}, Properties of bounded orthogonal spline bases. In \emph{Approximation theory (Papers, VIth Semester, Stefan Banach Internat. Math. Center, Warsaw, 1975)},  pp. 197--205, Banach Center Publ., 4, Warsaw, 1979.



\bibitem{ruckle} \textsc{W.H. Ruckle}, On the classification of biorthogonal sequences. Canadian J. Math {\bf 26} (1974), 721--733.

	
\bibitem{SY}
\textsc{C. Shao, P. Ye}, \emph{Lebesgue constants for Chebyshev thresholding greedy algorithms}, Journal of Inequalities and Applications (2018), Paper No. 102, 23 pp.

\bibitem{SingerI}
\textsc{I. Singer}. Bases in Banach spaces I. Springer-Verlag, 1970.


\bibitem{SingerII}
\textsc{I. Singer}. Bases in Banach spaces II. Springer-Verlag, 1981.


%
\bibitem{Tem98trig}
\textsc{V. N. Temlyakov}, \emph{Greedy algorithm and n-term
trigonometric approximation}, Const.Approx. 14 (1998), 569--587.
%

\bibitem{Tem1}
\textsc{V.N. Temlyakov}, {\it Greedy approximation}. Cambridge University Press, 2011.

\bibitem{TemW}
\textsc{V. N. Temlyakov}, \emph{The best m-term approximation and greedy algorithms}, Adv. Comput. {\bf 8} (1998), 249--265.


\bibitem{Tem15}
\textsc{V. N. Temlyakov}, Sparse approximation with bases. Ed. by S. Tikhonov. Advanced Courses in Mathematics. CRM Barcelona. Birkhäuser-Springer, 2015.



\bibitem{TYY2}
\textsc{V. N. Temlyakov, M. Yang, P. Ye}, \emph{Lebesgue-type
inequalities for greedy approximation with respect to quasi-greedy
bases}, East J. Approx {\bf 17} (2011), 127--138.


\bibitem{weisz01} \textsc{F. Weisz}, On the Fej\'er means of bounded Ciesielski systems. Studia Math.  {\bf 146} (3) (2001), 227--243.


	\bibitem{Wo} \textsc{P. Wojtaszczyk},
	\emph{Greedy Algorithm for General Biorthogonal Systems}, Journal of
	Approximation Theory, 107, (2000), 293--314.
%

\end{thebibliography}

\vskip 1truemm
\end{document}